\documentclass[conference]{IEEEtran}
\IEEEoverridecommandlockouts   
\usepackage{amsmath,amssymb}
\usepackage[italian,english]{babel}
\usepackage{graphicx}
\usepackage[sort, numbers]{natbib}
\usepackage{epstopdf}
\newtheorem{theorem}{Theorem}

\newtheorem{remark}{Remark}
\newtheorem{assumption}{Assumption}
\newtheorem{proof}{Proof}

\begin{document}

\title{\LARGE Nonlinear Control of an AC-connected DC MicroGrid}
%%\title{Management of the Interconnection of Intermittent Photovoltaic Systems Through a DC Link and Storages}
%%\author{Alessio, Sabah, Gilney, Abdelkrim, Françoise, Elena, Marika}
%%\author{A. Iovine,}
%%\author{S. B. Siad,}
%%\author{G. Damm,}
%%\author{A. Benchaib,}
%%\author{F. Lamnabhi-Lagarrigue,}
%%\author{E. De Santis,}
%%\author{M. D. Di Benedetto}

%\author{A. Iovine\thanks{Alessio Iovine, Elena De Santis and Marika Di Benedetto are with the Center of Excellence DEWS, Department of Information Engineering, Computer Science and Mathematics, University of L'Aquila, Italy.  Email: alessio.iovine@graduate.univaq.it, \{elena.desantis, mariadomenica.dibenedetto\}@univaq.it. Corresponding author: Alessio Iovine. Permanent email: alessio.iovine@hotmail.com}, G. Damm\thanks{Gilney Damm is with IBISC - Universit\`e d'Evry Val d'Essonne, Evry, France. Email: gilney.damm@ibisc.fr}, E. De Santis, M. D. Di Benedetto}

\author{A. Iovine$^a$\thanks{$^a$ Alessio Iovine, Elena De Santis and Marika Di Benedetto are with the Center of Excellence DEWS, Department of Information Engineering, Computer Science and Mathematics, University of L'Aquila, Italy.  Email: alessio.iovine@graduate.univaq.it, \{elena.desantis, mariadomenica.dibenedetto\}@univaq.it.}, S. B. Siad$^b$, G. Damm$^b$\thanks{$^b$ Sabah B. Siad and Gilney Damm are with IBISC - Universit\`e d'Evry Val d'Essonne, Evry, France. Email: gilney.damm@ibisc.fr, siadsabah@yahoo.fr  \newline This work is partially supported iCODE project.}, E. De Santis$^a$, M. D. Di Benedetto$^a$}
\maketitle
%
%%\vspace{-1.50cm}
%

\begin{abstract}
New connection constraints for the power network (Grid Codes) require more flexible and reliable systems, with robust solutions to cope with uncertainties and intermittence from renewable energy sources (renewables), such as photovoltaic arrays. A solution for interconnecting such renewables to the main grid is to use storage systems and a Direct Current (DC) MicroGrid. A "Plug and Play" approach based on the "System of Systems" philosophy using distributed control methodologies is developed in the present work. This approach allows to interconnect a number of elements to a DC MicroGrid as power sources like photovoltaic arrays, storage systems in different time scales like batteries and supercapacitors, and loads like electric vehicles and the main AC grid. The proposed scheme can easily be scalable to a much larger number of elements.
\end{abstract}%

%\begin{IEEEkeywords}
%, grid stability, mixed MicroGrid,
%\end{IEEEkeywords}}
\renewcommand{\abstractname}{Keywords}
%\vspace{-0.3cm}
\begin{abstract}
DC/AC MicroGrid, Power generation control, Lyapunov methods, Grid stability
\end{abstract}%

%\nocite{Iovine2016DCmicrogrids}
\vspace{-0.2cm}
\section{INTRODUCTION}\label{Sec_Introduction}
\vspace{-0.1cm}
Renewable energy is the key for locally producing clean and inexhaustible energy to supply the world's increasing demand for electricity. Photovoltaic (PV) conversion of solar energy is a promising way to meet the growing demand for energy, and is the best fit in several situations \cite{Eltawil2010PVconnectedgridproblems}. However, its intermittent nature remains a real disability that can create voltage (or even frequency in the case of islanded MicroGrids) instability for large scale grids. In order to answer to the new constraints of connection to the network (Grid Codes) it is possible to consider storage devices \cite{Barton2004energystorage}, \cite{Krajacic20112073}; the whole system will be able to inject the electric power generated by photovoltaic panels (or other renewables) to the grid in a controlled and efficient way. As a consequence, it is necessary to develop a strategy for managing energy in relation to the load and the storages' constraints. Direct Current (DC) microgrids are attracting interest thanks to their ability to easily integrate modern loads, renewable sources and energy storages \cite{Piagi2006}, \cite{Iravani2007}, \cite{Guerrero2014LVDC}, \cite{Guerrero2013advancedcontrol} since most of them (like electric vehicles, batteries and photovoltaic panels) are naturally DC: therefore, in this paper a DC microgrid composed by a source, a load, two storages working in different time scales, and their connecting devices is considered. These microgrids need to interact with the already existing infrastructure, that is an Alternate Current (AC) grid: in this work connection of the DC microgrid with a main AC grid is described, and the dedicated interconnecting power device is considered.

The utilized approach is based on a "Plug and Play" philosophy: the global control will be carried out at local level by each actuator, according to distributed control paradigm. The controller is developed in a distributed way for stabilizing each part of the whole system, while performing power management in real time to satisfy the production objectives and assuring the stability of the interconnection to the main grid.

Control techniques for converters are a well known research field \cite{siraramiirez_silva-ortigoza_2006}, \cite{A_chen_damm_cdc_2014}, \cite{chen:hal-01159853}: nevertheless, it is common practice to consider the hypothesis to have full controllability of the system \cite{A_tahim_pagano_lenz_stramosk_2015}, \cite{A_bidram_davoudi_lewis_guerrero_2013} while in reality, due to technical reasons, this will probably not be true. For example, realistic DC/DC converters have an additional variable (a capacitor) on the source side \cite{Walker2004cascadedPVconnection}. When this capacitor is controlled, another one (the grid side capacitor) is left uncontrolled. As explained in \cite{Iovine2016DCmicrogrids}, this remaining dynamics is usually neglected by the assumption that it is connected (and implicitly stabilized) by an always stable strong main grid. Removing this assumption to consider a realistic grid implies that this dynamics needs to be taken into account when studying grid stability. In \cite{Iovine2016DCmicrogrids} the authors provide a rigourous stability analysis for a realistic DC MicroGrid. This work regards the connection of the aforementioned DC grid to an AC one, evaluating dynamics interaction when fulfilling request of a desired amount of active and reactive power from the AC grid. Stability conditions are evaluated for the interconnected case. %, as well as the
%In case of islanded grid (no power exchange between the two grids), the

The adopted control strategy is shown to work both in islanded mode than in grid-connected mode: the AC grid is seen as a controllable load, while the load directly connected to the DC grid is uncontrolled (it can be constant or time-varying, both problems are relevant \cite{Marx2012}, \cite{Hamache2014backstepping}). %The whole system provides protection against faults and suppresses interference, and has a positive impact on the behaviour of the complete electrical system. The final management system can be configurable and adaptable as needed.

This paper is organized as follows. In Section \ref{Sec_models} the model of the AC connected DC MicroGrid is introduced. Then in Section \ref{Sec_control_laws} the adopted control laws are introduced and stability requirements are proven to be satisfied. %Section \ref{Sec_physical_explanation} explains the necessity of the adopted analysis.
Section \ref{sec_simulation_results_energy} provides simulation results, while in Section \ref{sec_conclusion} conclusions are offered.
\vspace{-0.3cm}
\section{MicroGrid}\label{Sec_models}

%\subsection{Problem definition}\label{Sec_problem_definition}
%%
The reference framework is depicted in Figure \ref{Fig_microgrid_example}, where the DC microgrid connected to the main AC is represented. The targets would be to assure voltage stability in the DC grid while correctly feeding power to the load; if possible, power is also provided to the main grid regulating both active and reactive power. To each component of the microgrid (PV array, battery, supercapacitor) a DC/DC converter is connected: their dynamical models are described in the following, as well as the AC/DC converter that connects the DC grid to the AC one.
%i

%
The whole control objective is split in several tasks; the first one is to extract the maximum available power from the photovoltaic array. This maximum power production is obtained calculating the duty cycle in order to fix the voltage of the capacitor directly connected to the PV array to a given reference. %Backstepping theory is used to stabilize the DC/DC boost converter that connected the solar array to the DC network.

The focus then moves to the storage systems and their connection to the DC network. In this paper, two kinds of storage are considered: a battery, which purpose is to provide/absorb the power when needed, and a supercapacitor, which purpose is to stabilize the DC grid voltage in case of disturbances. DC/DC bidirectional converters are necessary to enable the two modes of functioning (charge and discharge). The battery is assimilated as a reservoir which acts as a buffer between the flow requested by the network and the flow supplied by the production sources, and its voltage is controlled by the DC/DC current converter. With this structure, the DC grid is able to provide a continuous supply of good quality energy. The model introduced in \cite{Lifshitz2015Battery} is here used for the supercapacitor. %Again, backstepping theory is used for providing stability.

Considering the availability of power, it can be provided to the main AC grid. An AC/DC converter is dedicated to manage the interconnection of the two grids; its purpose is to provide a desired amount of power to the AC grid, regulating both active and reactive power.

The converters present in this system must, in a distributed way, keep the stability of the mixed DC/AC network interconnecting all parts. The final management system can be configurable and adaptable as needed.
\vspace{-0.2cm}
\subsection{Assumptions}\label{subsec_assumptions}
In this paper two main assumptions are made: the first one is the existence of a higher level controller which provide references to be accomplished by the local controllers \cite{Olivares2014Trendsmicrogrid}; the second one is about a proper sizing of each component of the microgrid in order to have feasible power balance.
\begin{assumption}
A higher level controller provides references for the local controllers: these references change every fixed time interval $T$ and concern the amount of power needed for the next time interval and the desired voltage value for the DC grid. The time interval $T$ is decided by the high level controller according to the computational time needed for calculations. These references are about the desired voltage to impose to the PV array and to the battery to obtain the needed amount of power, $V_1^*$ and $V_4^*$ respectively, the desired voltage value for the DC grid, $V_9^*$, and the desired currents to have a proper amount of active and reactive power to provide to the AC grid, $I_{d}^*$ and $I_{q}^*$. The references must be able to take into account a proper charge/discharge rate power for the supercapacitor.
\end{assumption}
\begin{assumption}
The sizing of the photovoltaic array is performed according to total energy needed into a whole day, and the sizing of the battery and the supercapacitor are performed according to the energy balance in a $T$ time step, needed for selecting a new reference.
\end{assumption}

\vspace{-0.1cm}
\subsection{Grid modeling}
%In this Section the considered framework depicted in Figure \ref{Fig_microgrid_example} is described. The PV array, battery and supercapacitor are each one connected to the DC grid by a DC/DC converter, while the connection between the AC and DC grids is described by the AC/DC converter.
%In this Section the considered framework depicted in Figure \ref{Fig_microgrid_example} is described.
Here the circuital representation and the mathematical models are given, based on power electronics averaging technique for the DC/DC converters \cite{sanders_noworolski_liu_verghese_1991}, \cite{middlebrook_cuk_1977}, while the model for the AC/DC converter is given by \cite{Blasko1997ACDCconverter}, \cite{Dinavahi2009nonlinearcontrolVSC}.

%
%\nocite{Lindberg1996PWMeACDC}

\begin{figure}
%%%\vspace{-1cm}
  \centering\includegraphics[width=1\columnwidth]{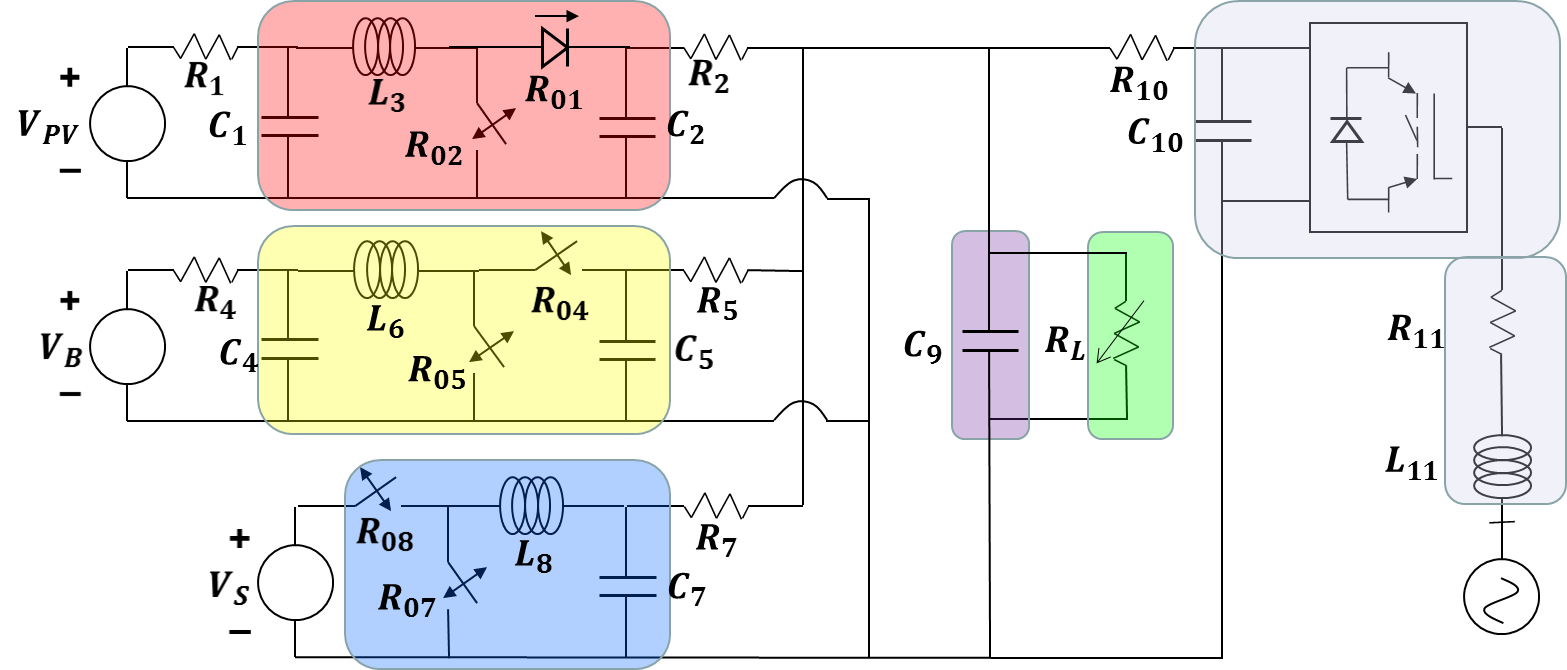}\\
  \caption{The considered framework: the red area describes the DC/DC boost converter connected to the PV array, while the yellow one the DC/DC bidirectional one connected to the battery. The DC/DC buck converter connecting the supercapacitor is in the blue area, while the AC/DC converter connecting the DC grid to the AC one is in the grey area. The uncontrollable load is in the green area, while the capacitor representing the DC grid in the violet one.}\label{Fig_microgrid_example}
  \vspace{-0.5cm}
\end{figure}
The resulting DC MicroGrid system is the composition of circuits in Figure \ref{Fig_microgrid_example}, generating the model introduced in the following:
\small
\begin{equation}\label{eq_all_syste}
x = \left[
\begin{array}{c}
x_{1} \\
%\alpha_1 \\
x_{2} \\
x_{3} \\
%\alpha_3 \\
x_{4} \\
%\alpha_4 \\
x_{5} \\
x_{6} \\
%\alpha_6 \\
x_{7} \\
%\alpha_7 \\
x_{8} \\
%\alpha_8 \\
x_{9} \\
x_{10} \\
x_{11} \\
x_{12}
\end{array}%
\right] =\left[
\begin{array}{c}
C_{1}\text{ capacitor voltage} \\
%\alpha_1 \text{ integral error} \\
C_{2}\text{ capacitor voltage} \\
L_{3}\text{ inductor current} \\
%\alpha_3 \text{ integral error} \\
C_{4}\text{ capacitor voltage} \\
%\alpha_4 \text{ integral error} \\
C_{5}\text{ capacitor voltage} \\
L_{6}\text{ inductor current} \\
%\alpha_6 \text{ integral error} \\
C_{7}\text{ capacitor voltage} \\
%\alpha_7 \text{ integral error} \\
L_{8}\text{ capacitor voltage} \\
%\alpha_8 \text{ integral error} \\
C_{9}\text{ capacitor voltage}\\
C_{10}\text{ capacitor voltage}\\
L_{11}\text{ inductor current }i_d \\
L_{11}\text{ inductor current }i_q
\end{array}%
\right]
\end{equation}
\normalsize
\footnotesize
\begin{equation}\label{EQ_DC_microgrid}
\begin{array}{l}
\begin{cases}
\dot{x}_{1}=-\frac{1}{R_{1}C_{1}}x_{1} - \frac{1}{C_{1}}x_{3} + \frac{1}{R_{1}C_{1}}V_{PV}\\ &\mbox{} \\
\dot{x}_{2}=-\frac{1}{R_{2}C_{2}}x_{2} + \frac{1}{C_{2}}x_{3} - \frac{1}{C_{2}}u_{1}x_{3} + \frac{1}{R_{2}C_{2}}x_{9} \\ &\mbox{}
\\
\dot{x}_{3}=\frac{1}{L_{3}}x_{1}-\frac{1}{L_{3}}x_{2}-\frac{R_{01}}{L_{3}}x_{3}+\frac{1}{L_{3}}x_{2}u_{1}+\frac{R_{01}-R_{02}}{L_{3}}x_{3}u_{1}
\\ &\mbox{} \\
\dot{x}_{4}=-\frac{1}{R_{4}C_{4}}x_{4}-\frac{1}{C_{4}}x_{6}+\frac{1}{R_{4}C_{4}}V_{B} \\ &\mbox{} \\
\dot{x}_{5}=-\frac{1}{R_{5}C_{5}}x_{5}+\frac{1}{C_{5}}x_{6}-\frac{1}{C_{5}}u_{2}x_{6}+\frac{1}{R_{5}C_{5}}{x}_{9}  \\ &\mbox{} \\
\dot{x}_{6}=\frac{1}{L_{6}}x_{4}-\frac{1}{L_{6}}x_{5}-\frac{R_{04}}{L_{6}}x_{6}+\frac{1}{L_{6}}x_{5}u_{2}\\ &\mbox{} \\
\dot{x}_{7}=-\frac{1}{R_{7}C_{7}}x_{7}-\frac{1}{C_{7}}x_{8}+\frac{1}{R_{7}C_{7}}{x}_{9}  \\ &\mbox{} \\
\dot{x}_{8}=\frac{1}{L_{8}}V_{S}u_3 -\frac{R_{08}}{L_{8}}x_{8} -\frac{1}{L_{8}} x_7  \\ &\mbox{} \\
\dot{x}_{9} = \frac{1}{C_{9}}\left(\frac{1}{R_{2}}(x_{2}-{x}_{9}) + \frac{1}{R_{5}}(x_{5}-{x}_{9})\right) + \\ &\mbox{} \\
\:\:\:\:\:\:\:+ \frac{1}{C_{9}}\left(\frac{1}{R_{7}}(x_{7}-{x}_{9})+ \frac{1}{R_{10}}(x_{10}-{x}_{9}) - \frac{1}{R_L }{x}_{9}\right)\\ &\mbox{} \\
\dot{x}_{10}=-\frac{3}{2C_{10}}\frac{1}{x_{10}}\left(v_{ld}x_{11}+v_{lq}x_{12} \right)+\frac{1}{R_{10}C_{10}}\left(x_{9}-x_{10}\right)  \\ &\mbox{} \\
\dot{x}_{11}=\frac{1}{L_{11}}\left[-R_{11}x_{11}+\omega x_{12}+ \frac{1}{2}{x}_{10}u_4 - v_{ld} \right]\\ &\mbox{} \\
\dot{x}_{12}=\frac{1}{L_{11}}\left[-R_{11}x_{12}-\omega x_{13}+ \frac{1}{2}{x}_{10}u_5 - v_{lq} \right]
\end{cases}
\end{array}%
\end{equation}
\normalsize
\small%\footnotesize
\begin{equation}
\dot{x}(t) = f(x(t)) + g(x(t),u(t),d(t))+d(t)
\end{equation}
\normalsize
\small%\footnotesize
\begin{equation}\label{eq_all_syste_u}
u = \left[
\begin{array}{ccccc}
u_{1} & u_{2} & u_{3} & u_{4} & u_{5}
%\alpha_1 \\
\end{array}%
\right]^T
\end{equation}
\normalsize
\small%\footnotesize
\begin{equation}\label{eq_all_syste_d}
d = \left[
\begin{array}{ccccccc}
V_{PV} & V_{B} & V_{S} & {R_L} & v_{ld} & v_{lq} & \omega
%%\alpha_1 \\
% \\
%\\
%\frac{1}
\end{array}%
\right]^T
\end{equation}
\normalsize
where $C_1$, $C_2$, $C_4$, $C_5$, $C_7$, $C_9$, $C_{10}$, $R_1$, $R_2$, $R_{01}$, $R_{02}$, $R_4$, $R_5$, $R_{04}$, $R_{05}$, $R_7$, $R_{07}$, $R_{08}$, $R_{10}$, $R_{11}$, $R_{L}$, $L_{3}$, $L_{6}$, $L_{8}$, $L_{11}$, are known positive values of the capacitors, resistances and the inductor. $V_{PV}>0$, $V_B>0$ are constant positive values of the PV array and the battery voltages, while $V_S$ is a slowing time varying positive value known at each time $t$ representing the voltage of the supercapacitor. $v_{ld}\geq0$ and $v_{lq}\geq0$ are the AC relative voltages, while $\omega$ is the frequency of the AC grid. The control inputs in $u$ are the duty cycles of the converters.

We can now refer the references to the state variables: the voltage value $V_1^*$ to impose the maximum power tracking provides a reference $x_1^*$ for dynamics $x_1$, while $V_4^*$ and $V_9^*$ refer to $x_4^*$ and $x_9^*$. The references $I_{d}^*$ and $I_{q}^*$ (relative to the desired active and reactive powers provided to the AC grid) for the currents are related to the dynamics $x_{11}$ and $x_{12}$, $x_{11}^*$ and $x_{12}^*$ respectively.

\section{Controllers}\label{Sec_control_laws}

Considering the hypothesis in Section \ref{subsec_assumptions}, given the constant values of the resistance $R_L$, voltages $x_1^*$ and $x_4^*$ that allow power balance in steady-state with respect to the demanded currents $x_{11}^*$ and $x_{12}^*$ at the desired voltage grid $x_9^*$, it is possible to state that:

\begin{theorem}\label{theo_}
Control inputs $u_1$, $u_2$, $u_3$, $u_4$, $u_5$ exist such that the system in (\ref{EQ_DC_microgrid}) is asymptotically stable in closed loop around the equilibrium point $x^e$;

\begin{equation}\label{eq_equilibrium}
x^e = \left[
\begin{array}{c}
x_{1}^e \\
x_{2}^e \\
x_{3}^e \\
x_{4}^e \\
x_{5}^e \\
x_{6}^e \\
x_{7}^e \\
x_{8}^e \\
x_{9}^e \\
x_{10}^e \\
x_{11}^e \\
x_{12}^e
\end{array}%
\right] =\left[
\begin{array}{c}
x_1^*\\
x_{2}^* \\
\frac{1}{R_{1}}(V_{PV} - x_{1}^*) \\
x_{4}^* \\
x_5^*\\
\frac{1}{R_{4}}(V_{B} - x_{4}^*)  \\
x_9^* \\
0 \\
x_9^*\\
x_{10}^*\\
x_{11}^*\\
x_{12}^*
\end{array}%
\right]
\end{equation}
\end{theorem}

\begin{proof}
The proof is based on a composition of Lyapunov functions, a methodology described in \cite{kundur2004powerstability}. Indeed the control inputs
\small%\footnotesize
\begin{align}\label{formula_control_PV}
u_1 &= \frac{1}{{x_{2}+(R_{01}-R_{02})x_3}} \left[-x_{1} + x_{2} + R_{01}x_3 -L_3 v_1 \right]
\end{align}%
\normalsize
\small%\footnotesize
\begin{align}\label{formula_control_BAT}
u_{2}&= \frac{1}{x_{5}}\left(-x_{4}+ x_{5} + R_{04}x_{6} + L_6v_2\right)
\end{align}
\normalsize
\small%\footnotesize
\begin{equation}\label{eq_acdc_converter_u4}
u_4 = 2 \frac{1}{x_{10}}\left[v_{ld} + R_{11}x_{11} - \omega x_{12} - L_{11}v_4\right]
\end{equation}
\normalsize%
\small%\footnotesize
\begin{equation}\label{eq_acdc_converter_u5}
u_5 = 2 \frac{1}{x_{10}}\left[v_{lq} + R_{11}x_{12} + \omega x_{11} - L_{11}v_5\right]
\end{equation}
\normalsize%
with
\small%\footnotesize
\begin{align}\label{formula_control_PV_v1}
v_1 &=  K_3(x_3 - z_3) + \overline{K}_3\alpha_3  - C_1\overline{K}_1K_{1}^\alpha(x_1 - x_1^*)+\\
\nonumber&+\left(C_1K_1 - \frac{1}{R_1} \right)(K_1(x_1 - x_1^*)+ \overline{K}_1\alpha_1)
\end{align}
\begin{equation}\label{formula_eqpoint_x3_1}
z_3 = \frac{1}{R_{1}}(V_{PV} - x_{1}) + C_1K_1(x_1 - x_1^*) + C_1\overline{K}_1\alpha_1
\end{equation}
\normalsize
\small%\footnotesize
\begin{equation}\label{formula_alpha1}
\dot{\alpha}_1 = K_{1}^\alpha(x_1 - x_1^*) \:\:\:\:\:\:\:\: \dot{\alpha}_3 = K_{3}^\alpha(x_3 - z_3)
\end{equation}
\normalsize%
\small%\footnotesize
\begin{align}\label{formula_control_BAT_v2}
v_{2}=& -  K_6(x_6 - z_6) - \overline{K}_6 \alpha_6 + \overline{K}_4K_4^\alpha\left( x_4 -x_4^* \right) +\\
\nonumber & - \left( C_4 K_4 -\frac{1}{R_4}\right)\left( K_4(x_4 - x_4^*) + \overline{K}_4\alpha_4 \right)
\end{align}
\normalsize
\small%\footnotesize
\begin{equation}\label{Eq_battery_z6value}
z_6 = \left(\frac{1}{R_{4}}(V_{B}-x_{4}) + C_4K_4(x_4-x_4^*) +C_4\overline{K}_4\alpha_4\right)
\end{equation}
\normalsize
\small%\footnotesize
\begin{equation}\label{formula_alpha4}
\dot{\alpha}_4 = K_{4}^\alpha(x_4 - x_4^*) \:\:\:\:\:\:\:\:\:\dot{\alpha}_6 = K_{6}^\alpha(x_6 - z_6)
\end{equation}
\normalsize%
\small%\footnotesize
\begin{equation}\label{eq_acdc_converter_v4}
v_4 = K_{11}\left(x_{11} - x_{11}^* \right) + \overline{K}_{11}{\alpha}_{11}
\end{equation}
\normalsize%
\small%\footnotesize
\begin{equation}\label{eq_acdc_converter_v5}
v_5 = K_{12}\left(x_{12} - x_{12}^* \right) + \overline{K}_{12}{\alpha}_{12}
\end{equation}
\normalsize%
\small%\footnotesize
\begin{equation}\label{formula_alpha1314}
\dot{\alpha}_{11} = K_{11}^{\alpha}\left(x_{11} - x_{11}^* \right)\:\:\:\:\:\:\:\: \dot{\alpha}_{12} = K_{12}^{\alpha}\left(x_{12} - x_{12}^* \right)
\end{equation}
\normalsize
where the $\alpha_i$, $i=1,3,4,6,11,12$, are integral terms assuring zero error in steady state and the positive gains $K_i$, $\overline{K}_i$,  $K_{i}^\alpha$, $i=1,3,4,6,11,12$, are properly chosen, provide a Proportional Integral (PI) control action that feedback linearizes the dynamics $x_i$, $i=1,3,4,6,11,12$, providing asymptotic stability. Then it is possible to calculate positive definite Lyapunov functions $V_{1,3}$, $V_{4,6}$, $V_{11,12}$, such that their time derivative are negative definite. The Lyapunov function for the entire system is therefore selected as
%
%\small%\footnotesize
\begin{equation}\label{eq_lyap_entire}
V = V_{1,3} + V_{4,6} + V_{11,12} + V_{2,5,9,10} + V_{7,8}
\end{equation}
\normalsize
where $V_{7,8}$ and $V_{2,5,9,10}$ need to be described. We start from the last one, that describes the interconnection among the subsystems:
%
%\footnotesize
%\begin{equation}\label{Eq_interconnected_LyapunovACDC}
%  V_{2,5,9,10} = \frac{C_2}{2}e_2^2 + \frac{C_5}{2}e_5^2 + \frac{C_{10}}{2}e_{10}^2 + \frac{C_9}{2}x_9^2 %+ \frac{1}{2}\alpha_9^2
%\end{equation}
%\normalsize
%
\footnotesize
\begin{equation}\label{Eq_interconnected_LyapunovACDC}
  V_{2,5,9,10} = \frac{C_2}{2}(x_2-x_2^*)^2 + \frac{C_5}{2}(x_5-x_5^*)^2 + \frac{C_{10}}{2}(x_{10}-x_{10}^*)^2 + \frac{C_9}{2}x_9^2 %+ \frac{1}{2}\alpha_9^2
\end{equation}
\normalsize
%
%
%\footnotesize
%\begin{align}\label{Eq_interconnected_Lyapunov_dotACDC}
%  \dot{V}_{2,5,9,10} &= -\frac{1}{R_{2}}e_2^2 -\frac{1}{R_{5}}e_5^2 -\frac{1}{R_{10}}e_{10}^2 -\frac{1}{R_{7}}(x_9-x_9^*)^2\leq 0% - \alpha_9^2
%\end{align}
%\normalsize
According to the calculation of its time derivative, the dynamics $x_7$ can be used as control variable to obtain a negative semidefinite derivative; indeed, we can properly select a reference $z_7$ for $x_7$ such that
%
%\footnotesize
%\begin{align}\label{Eq_interconnected_Lyapunov_dotACDC}
%  \dot{V}_{2,5,9,10} &= -\frac{1}{R_{2}}(x_2-x_2^*)^2 -\frac{1}{R_{5}}(x_5-x_5^*)^2 -\frac{1}{R_{10}}(x_{10}-x_{10}^*)^2 -\frac{1}{R_{7}}(x_9-x_9^*)^2\leq 0% - \alpha_9^2
%\end{align}
%\normalsize
%
\small%\footnotesize
\begin{align}\label{Eq_interconnected_Lyapunov_dotACDC}
  \dot{V}_{2,5,9,10} &= -\frac{1}{R_{2}}(x_2-x_2^*)^2 -\frac{1}{R_{5}}(x_5-x_5^*)^2 +\\
  \nonumber &-\frac{1}{R_{10}}(x_{10}-x_{10}^*)^2 -\frac{1}{R_{7}}(x_9-x_9^*)^2\leq 0% - \alpha_9^2
\end{align}
\normalsize
To prove asymptotic stability the set $\Omega$ is considered: it is the largest invariant set of the set $E$ of all points where the Lyapunov function is not decreasing. $\Omega$ contains an unique point; then applying LaSalle's theorem, asymptotic stability of the equilibrium point can be established.
%
%\footnotesize
%\begin{align}\label{eq_omega _forsemidefACDC}
%  \Omega =& \{(e_2,e_5,e_{10},x_9) \: : \: x_2=x_2^*, \: x_5=x_5^*, \: x_{10}=x_{10}^*, \: x_9=x_9^* \} = \\
%\nonumber  =& \{(0,0,0,x_9^*)\}
%\end{align}
%\normalsize
%
\footnotesize
\begin{align}\label{eq_omega _forsemidefACDC}
  \Omega =& \{(x_2,x_5,x_{10},x_9) \: : \: x_2=x_2^*, \: x_5=x_5^*, \: x_{10}=x_{10}^*, \: x_9=x_9^* \} = \\
\nonumber  =& \{(x_2^*,x_5^*,x_{10}^*,x_9^*)\}
\end{align}
\normalsize
To impose the reference $z_7$ backstepping technique is used: a reference $z_8$ for the dynamics $x_8$ is selected in order to force the dynamics of $x_7$ to track the reference $z_7$, and a proper control law $u_3$ is calculated for the convergence of $x_8$ to $z_8$. The following Lyapunov function can be used to determine the control law:
\small%\footnotesize
\begin{equation}\label{EQ_Lyapunov_78}
V_{7,8} = \frac{1}{2}(x_7 - z_7)^2 + \frac{1}{2}(x_8 - z_8)^2
\end{equation}%
\normalsize
where $K_7>0$, $K_8>0$,
\small%\footnotesize
\begin{equation}\label{Eq_supercapacitor_z8value}
z_{8} = \frac{1}{R_{7}}(x_{9} - x_7) + C_7 K_7 (x_{7} - z_7) - C_7\dot{z}_7
\end{equation}
\normalsize
With the control law defined in the following
\small%\footnotesize
\begin{align}\label{formula_control_SCtheo}
u_3 =& \frac{1}{V_{S}}\left[x_7 + R_{08}x_8 + L_8 \dot{z}_{8} - L_8 v_3\right]
\end{align}
\normalsize
where
%
%\vspace{-0.3cm}
\small%\footnotesize
\begin{align}\label{formula_control_SCtheo_v3}
v_3 =&  K_8(x_8 - z_8) %+ K_7\left(K_7 C_7 - \frac{1}{R_7}\right)(x_7 - z_7)
\end{align}
\normalsize
\small%\footnotesize
\begin{equation}\label{Eq_supercapacitor_z8value}
\dot{z}_{8} = \frac{1}{R_{7}}\dot{x}_{9} - C_7 K_7 \dot{z}_7- K_7\left(K_7 C_7 - \frac{1}{R_7}\right)(x_7 - z_7) - C_7\ddot{z}_7
\end{equation}
\normalsize
the considered Lyapunov function in (\ref{EQ_Lyapunov_78}) has a negative definite time derivative:
%
%\small%\footnotesize
\begin{equation}\label{EQ_Lyapunov_78_dot_}
\dot{V}_{7,8} = -K_7(x_7 - z_7)^2 - K_8(x_8 - z_8)^2
\end{equation}%
\normalsize
Then the Lyapunov function $V$ in (\ref{eq_lyap_entire}) has the following time derivative that ensures asymptotic stability:
%
%\small%\footnotesize
\begin{equation}\label{}
  \dot{V} = \dot{V}_{1,3} + \dot{V}_{4,6} + \dot{V}_{11,12} + \dot{V}_{7,8}+ \dot{V}_{2,5,9,10} \leq 0
\end{equation}
\normalsize
\end{proof}
\begin{remark}
In accordance to the equilibria in (\ref{eq_equilibrium}), the power balance requirement is satisfied.
\end{remark}
The value of $z_7$ in steady-state can be written as
\footnotesize
\begin{align}\label{eq_z7_equilibrium}
  \frac{x_7^*}{R_{7}} &= \frac{{x}_{9}^*}{R_L} - \frac{1}{R_{2}}( x_2^*-{x}_{9}^*) - \frac{1}{R_{5}}( x_5^*-{x}_{9}^*) - \frac{1}{R_{10}}( x_{10}^*-{x}_{9}^*)+\frac{x_9^*}{R_{7}} %+ x_9^*
\end{align}
\normalsize
%
%\footnotesize
%\begin{align}\label{}
%  x_7^* &= R_7 \left[\frac{1}{R_L}{x}_{9}^* - \frac{1}{R_{2}}( x_2^*-{x}_{9}^*) +\\
%  \nonumber &- \frac{1}{R_{5}}( x_5^*-{x}_{9}^*) - \frac{1}{R_{10}}( x_{10}^*-{x}_{9}^*)\right] + x_9^*
%\end{align}
%\normalsize
%
where the values of $x_2^*$, $x_5^*$ and $x_{10}^*$ depends on the given equilibrium values and on the control inputs. At the equilibrium, the following condition is satisfied:
\small%\footnotesize
\begin{align}\label{eq_cond_equilibrium}
\frac{{x}_{9}^*}{R_L} = \frac{1}{R_{2}}( x_2^*-{x}_{9}^*) + \frac{1}{R_{5}}( x_5^*-{x}_{9}^*) + \frac{1}{R_{10}}( x_{10}^*-{x}_{9}^*) % + \frac{1}{R_{7}}( x_7^*-{x}_{9}^*)
\end{align}
\normalsize
Condition (\ref{eq_cond_equilibrium}) describes the power balance when all the "bricks" fit their target to provide/take the right amount of power; implicitly, it can be stated that $x_7^* = x_9^*$.

Theorem \ref{theo_} refers to an unconstrained problem: due to the physics of the devices and of the meaning of the controllers,  all the control inputs are bounded. The bounds we must consider are: $u_1\in[0,1]$, $u_2\in[0,1]$, $u_3\in[0,1]$, $\sqrt{u_4^2 + u_5^2}\leq1$ (see \cite{siraramiirez_silva-ortigoza_2006}, \cite{A_chen_damm_cdc_2014}).

The domain of operation for the state variables is restricted (e.g. $x_9\in[x_9^m,x_9^M]$, where $x_9^m\geq max(V_{PV},V_B)$ and $x_9^M\leq V_S$). However in \cite{Iovine2016DCmicrogrids} we proved that the set of initial states such that the above defined input constraints are satisfied, for any evolution of the controlled system, contains the equilibrium in its interior.

%According to the defined bounds and to the fact that the domain of operation for the variables is restricted to a narrow area (for example, $x_9\in[x_9^m,x_9^M]$, where $x_9^m\geq max(V_{PV},V_B)$ and $x_9^M\leq V_S$), we can state that the set of states starting from which the state trajectories of the controlled system are such that the input constraints are satisfied is not empty and contains the equilibrium in its interior, as already stated in \cite{Iovine2016DCmicrogrids}.

%Finally, we have shown that exist a Lyapunov function $V$ such as in () that has a semidefinite negative $\dot{V}$: applying LaSalle's theorem we can prove asymptotical convergence. Thereby, as $x_9\rightarrow x_9^*=x_9^e$, then $x_2\rightarrow x_2^e$, $x_5\rightarrow x_5^e$, $z_7\rightarrow x_7^*=x_9^*=x_7^e$, $z_8\rightarrow x_8^e=0$. Furthermore, $x_1\rightarrow x_1^*$ and $x_4\rightarrow x_4^*$ imply that $x_3\rightarrow x_3^e$ and $x_6\rightarrow x_6^e$.

\section{Simulations}\label{sec_simulation_results_energy}
In this section we present some simulations that show the results obtained using the proposed control inputs. Matlab has been used for obtaining such simulations. The values of the parameters are depicted in Tables \ref{table_PV}.

\begin{table}
  \centering
    \caption{Grid parameters.}\label{table_PV}
         \vspace{-0.1cm}
    \begin{tabular}{ | l | l | l | l |}
    \hline
    Parameter  & Value & Parameter  & Value \\ \hline
    $C_1$  & 100 mF & $L_{3}$ & 33 mH \\ \hline
    $C_2$ & 10 mF & $R_{01}$ & 10 m$\Omega$ \\ \hline
    $R_1$ & 100 m$\Omega$ & $R_{02}$ & 10 m$\Omega$ \\ \hline
    $R_2$ & 100 m$\Omega$ & $C_4$   & 100 mF \\ \hline
   %     & & & \\ \hline
    $C_5$ & 10 mF & $R_{04}$ & 10 m$\Omega$ \\ \hline
    $R_4$ & 100 m$\Omega$ & $R_{05}$ & 10 m$\Omega$ \\ \hline
    $R_5$ & 10 m$\Omega$ & $L_{6}$   & 33 mH \\ \hline
        $C_7$  & 10 mF & $L_{8}$ & 3.3 mH \\ \hline
    $R_{07}$ & 10 m$\Omega$  & $R_{08}$ & 10 m$\Omega$ \\ \hline
 %   $R_7$ & 0.1 $\Omega$ & $R_{05}$ & 0.001 $\Omega$ \\ \hline
    $R_7$ & 100 m$\Omega$ & $C_9$  & 0.1 mF  \\ \hline
        $R_{10}$ & 100 m$\Omega$  & $L_{11}$  & 3.3 mH \\ \hline
           $R_{11}$ & 10 m$\Omega$ & $C_{10}$  & 680 $\mu F$  \\ \hline
           $f$ & 50 Hz &   &   \\
    \hline
    \end{tabular}
         \vspace{-0.4cm}
%  \vspace{-0.7cm}
\end{table}

The simulation target is to correctly feed a load while maintaining the grid stability, which means to ensure no large variation in the DC grid voltage. The load is composed by an uncontrollable part, which is a resistance, and a controllable one, which is the AC grid. The simulation time is twenty seconds. The fixed value $x_9^*$ for the DC grid voltage is selected as $x_9^*=1000$ $V$. The secondary controller provides the references to be reached each time interval (1 second); during that period the introduced control laws bring the devices to operate in the desired points. Figure \ref{Fig_PV_BAT_SC_currentsRL} depicts the uncontrollable load; it is possible to see its piecewise constant behavior until the simulation time of twelve seconds, when it starts to be time-varying. Figure \ref{Fig_voltagesx11x12} describes the currents related to the power demanded by the AC grid; the current related to reactive power is always kept to zero reference, while the active one is demanded in a time window of eleven seconds after three seconds. The choice to send zero reactive power is due to the fact that the remuneration of power production is solely based on active power; selection of providing both of them is feasible if needed, and is the case when fully controlling frequency and voltage of the AC MicroGrid.

%\vspace{-0.4cm}
%
\begin{figure}%[h!]
%\vspace{-0.2cm}
  \centering
  \includegraphics[width=1\columnwidth]{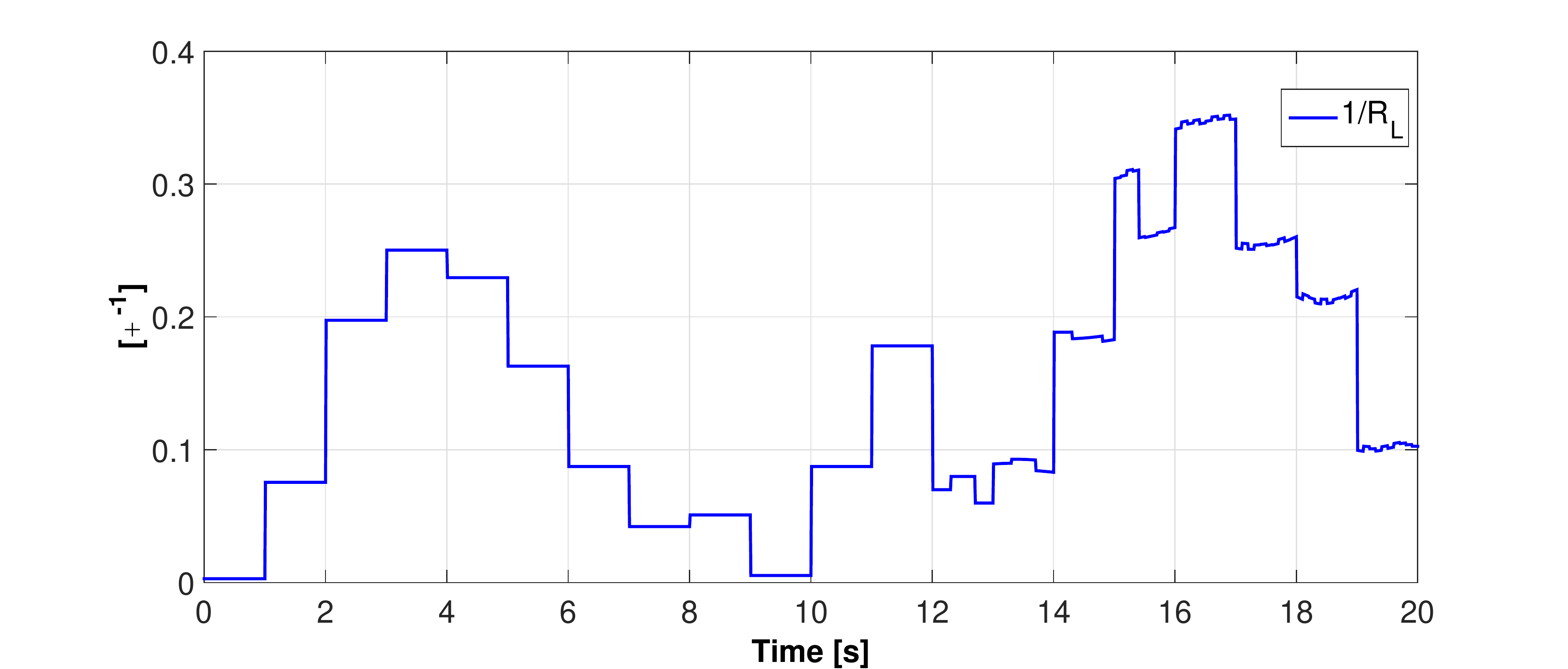}\\
  \vspace{-0.3cm}
  \caption{The load conductance $\frac{1}{R_L}$.}\label{Fig_PV_BAT_SC_currentsRL}
  \vspace{-0.7cm}
\end{figure}
%
%	\vspace{-0.5cm}
\begin{figure}%[h!]
	%\vspace{-0.2cm}
	\centering
	\includegraphics[width=1\columnwidth]{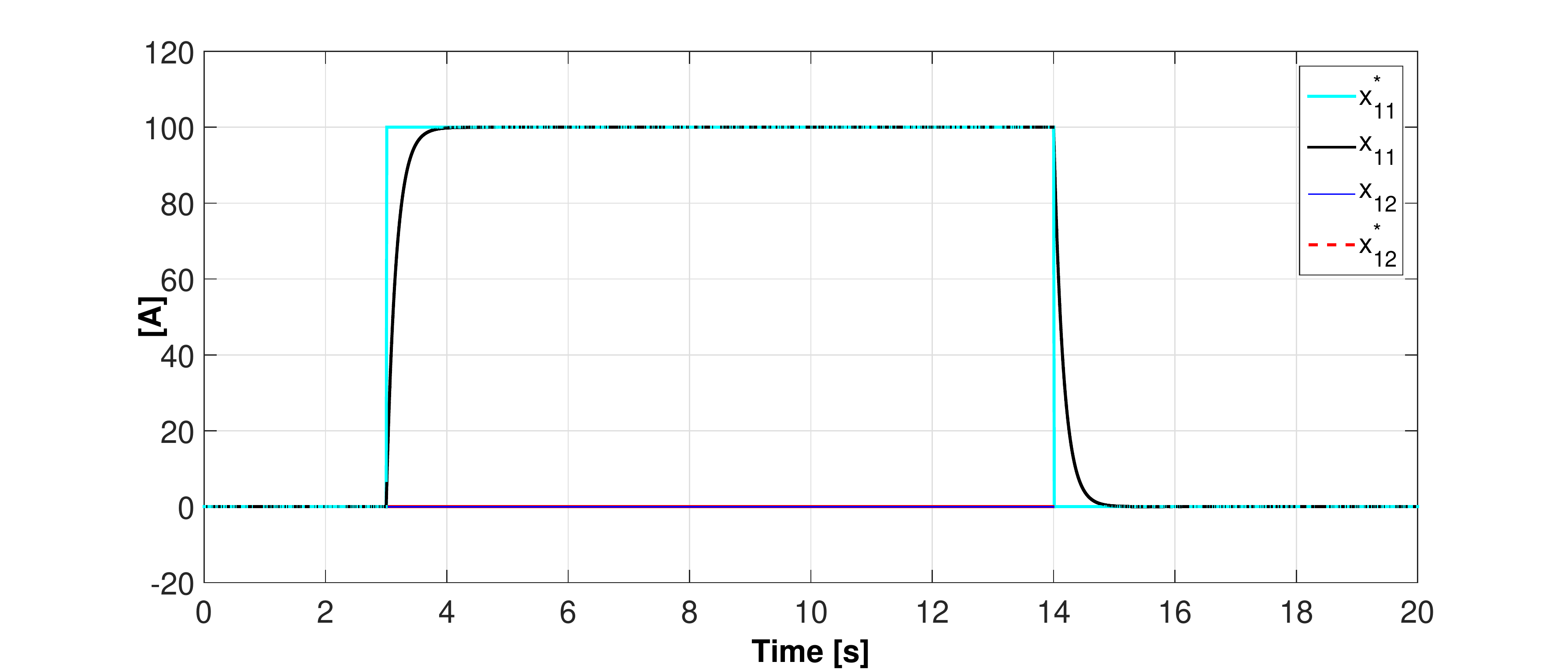}\\
	\vspace{-0.3cm}
	\caption{The currents in the inductance $L_{11}$ following their references: the current related to the active power ($x_{11}$) is the black line, tracking its cyan reference, while the one related to the reactive power ($x_{12}$) is depicted by the blue line which is tracking the red line representing its reference.}\label{Fig_voltagesx11x12}
	\vspace{-0.2cm}
\end{figure}
%
%	\vspace{-0.4cm}
\begin{figure}%[h!]
	%\vspace{-0.2cm}
	\centering
	\includegraphics[width=1\columnwidth]{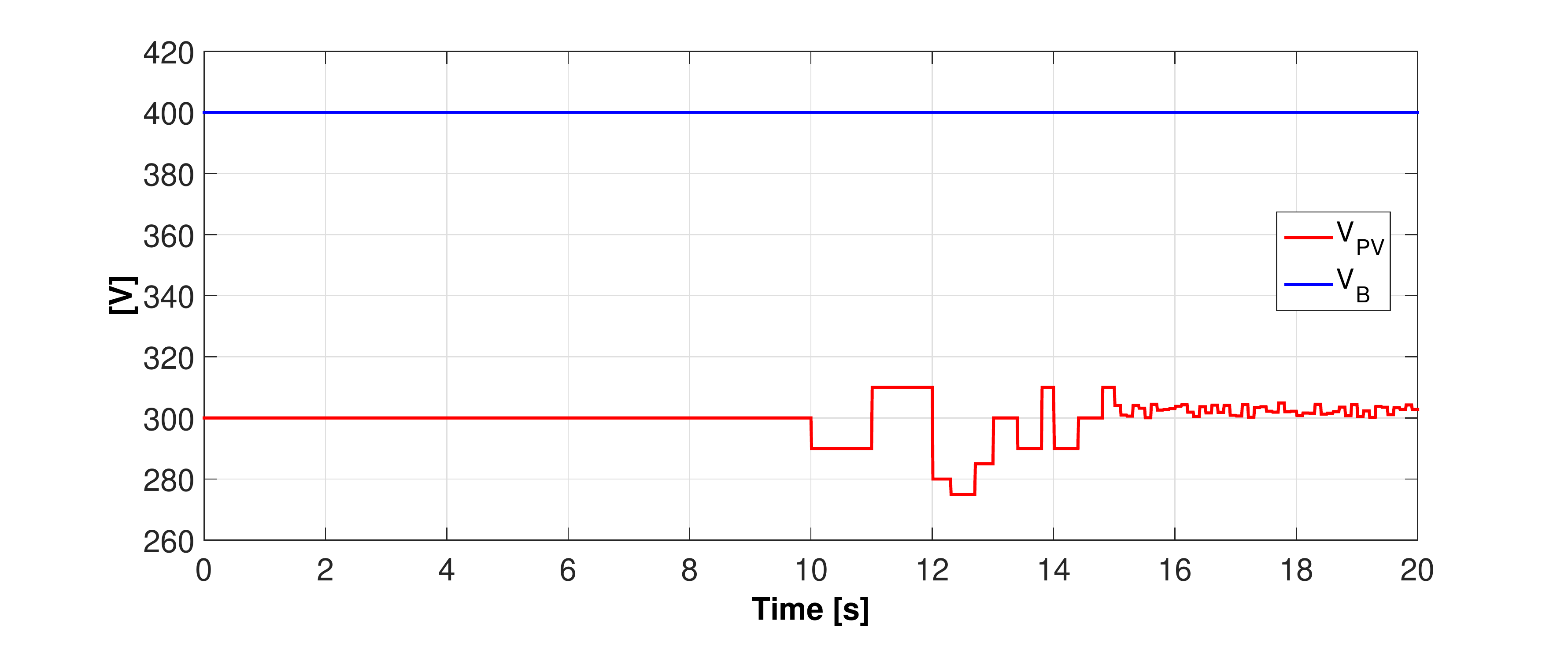}\\
	\vspace{-0.5cm}
	\caption[The voltages $V_{PV}$ and ${V_B}$]{The voltages of $V_{PV}$ (red line) and ${V_B}$ (blue line).}\label{Fig_PV_BAT_SC_voltagesVpvVb}
	\vspace{-0.4cm}
\end{figure}
%
%\vspace{-1cm}
%
%\vspace{-1cm}
\begin{figure}[h!]
	%\vspace{-0.2cm}
	\centering
	\includegraphics[width=1\columnwidth]{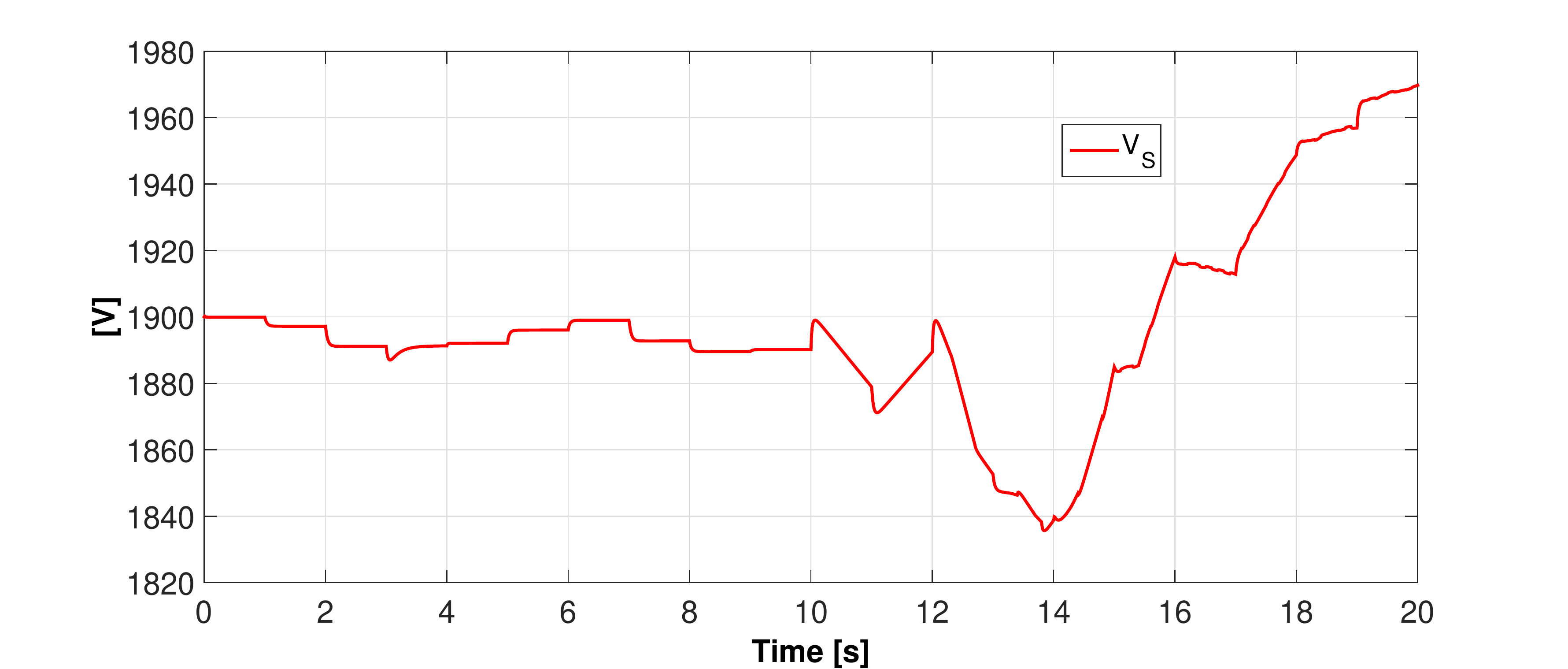}\\
	\vspace{-0.3cm}
	\caption{The voltage of the supercapacitor.}\label{Fig_PV_BAT_SC_voltagesVs}
	\vspace{-0.3cm}
\end{figure}
\begin{figure}[h!]
	%\vspace{-0.2cm}
	\centering
	\includegraphics[width=1\columnwidth]{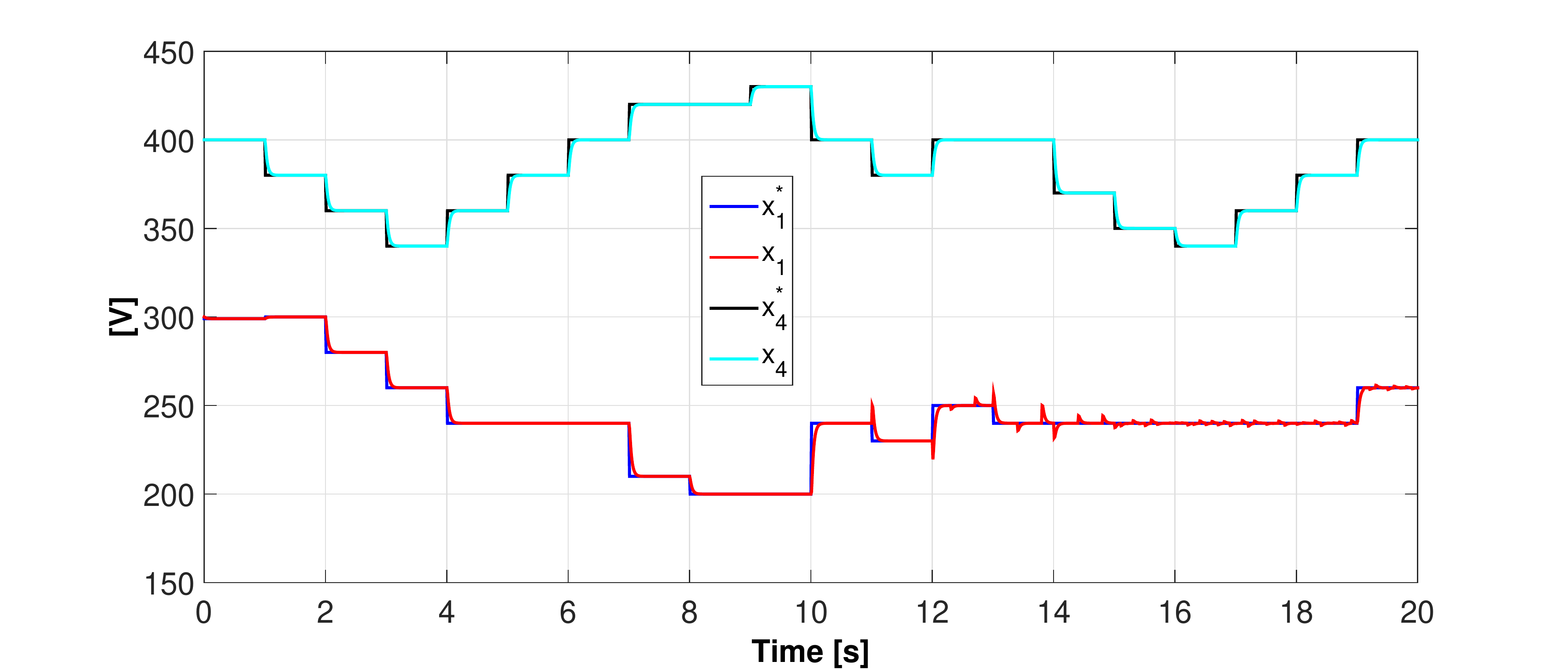}\\
	\vspace{-0.3cm}
	\caption[The voltages of ${C_1}$ and ${C_4}$]{The voltages of ${C_1}$ (red line) and ${C_4}$ (cyan line) following the desired references, the blue and black lines, respectively.}\label{Fig_PV_BAT_SC_voltagesx1x4}
	\vspace{-0.4cm}
\end{figure}
\begin{figure}[h!]
	%\vspace{-0.2cm}
	\centering
	\includegraphics[width=1\columnwidth]{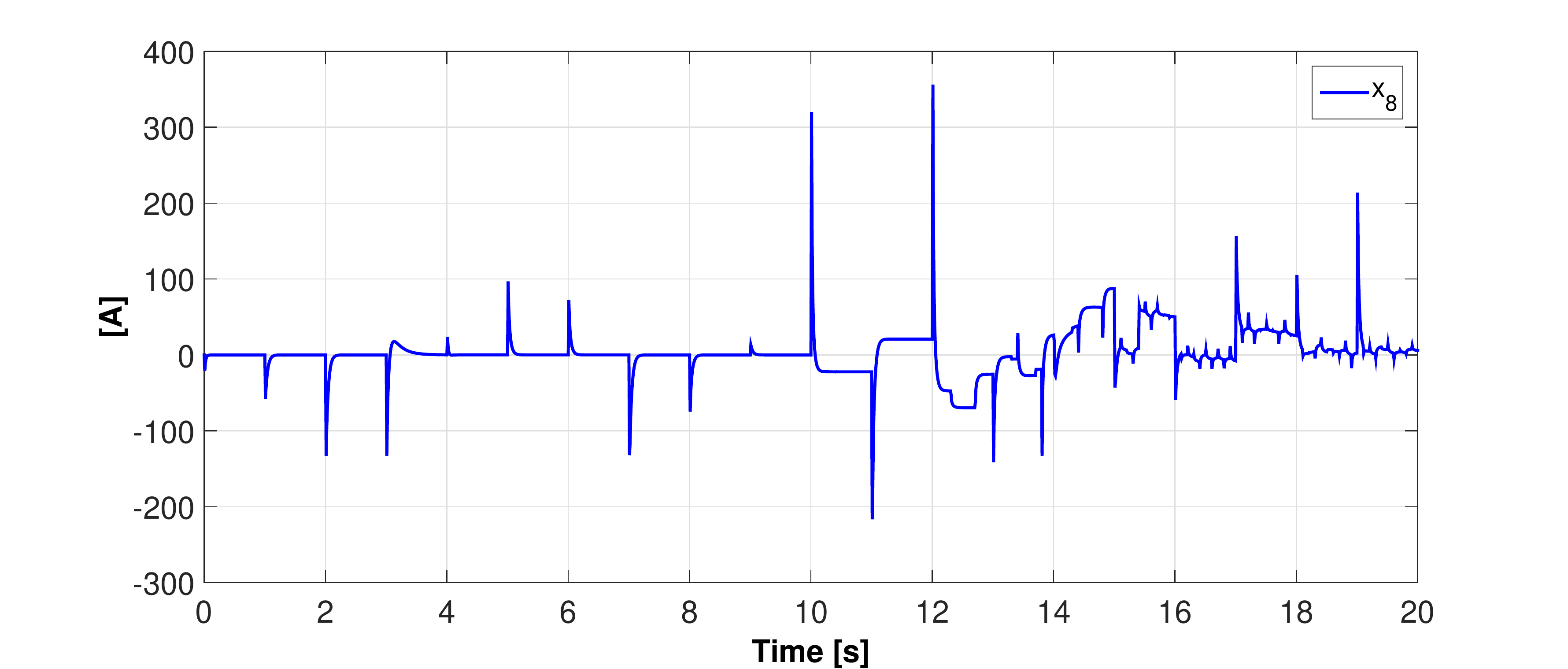}\\
	\vspace{-0.3cm}
	\caption{The current in the inductance $L_8$.}\label{Fig_PV_BAT_SC_currentsx8iL}
	\vspace{-0.2cm}
\end{figure}

%\vspace{-1cm}
%
The considered simulation framework contains different situations: as we have already seen, the load has step variations and a time varying part. Furthermore, disturbances acting on the PV voltage are considered (see Figure \ref{Fig_PV_BAT_SC_voltagesVpvVb}). Moreover, the case where the references do not fulfill the energy balance without the power provided by the supercapacitor is considered: there the supercapacitor will need to provide power during all the time window, and its reservoir will change over the time, as depicted in Figure \ref{Fig_PV_BAT_SC_voltagesVs}. Here we considered the voltage of the battery not to be affected by the current behavior; indeed a constant value is used to represent it because the battery is supposed to be sized in such a way that it is not affected by current dynamics over a time of twenty seconds.
The resulting current generated by the supercapacitor is then introduced in Figure \ref{Fig_PV_BAT_SC_currentsx8iL}; it is mainly due to the necessity of the supercapacitor to ensure DC grid voltage stability around the equilibrium value of $x_9^*$ when there is a mismatch due to the transient or disturbances, but in the second half of the simulation it is also providing the power needed to compensate power mismatch due to wrong reference choice by the higher level controller. Indeed, in accordance to the values of the load, the references $x_1^*$ and $x_4^*$ are obtained for the power balance target; if they are not sized for the power demand, the supercapacitor has also the described role.

%\vspace{-1cm}
%

%

As depicted in Figure \ref{Fig_PV_BAT_SC_voltagesx1x4}, the $C_1$ and $C_4$ capacitor voltages reach the desired values during the considered time step. Two different situations for the controllers are faced because the devices need two different treatments; we need from the PV array to start providing the highest level of power as soon as possible, while the battery needs to have a smooth behavior to preserve its life-time. The resulting voltage dynamics on the grid connected capacitors $C_2$, $C_5$, $C_{10}$, are modified by the current flow generated by the sources; all the dynamics are stable, as shown in Figure \ref{Fig_PV_BAT_SC_voltagesx2x5}. Their evolution is influenced by the value of the DC grid voltage, which is the capacitor $C_9$: its value over time is depicted in Figure \ref{Fig_PV_BAT_SC_currentsx7x9}. To satisfy stability constraints, in response to the load variations and to the missing power coming from the PV and the battery, the voltage of the capacitor $C_7$ reacts balancing the energy variation. It is possible to verify that the worst spikes correspond to a very high variations in the grid dynamics, but nevertheless they are less than 10\% of the values.
%

%\vspace{-1cm}
\begin{figure}%[h!]
%\vspace{-0.2cm}
  \centering
  \includegraphics[width=1\columnwidth]{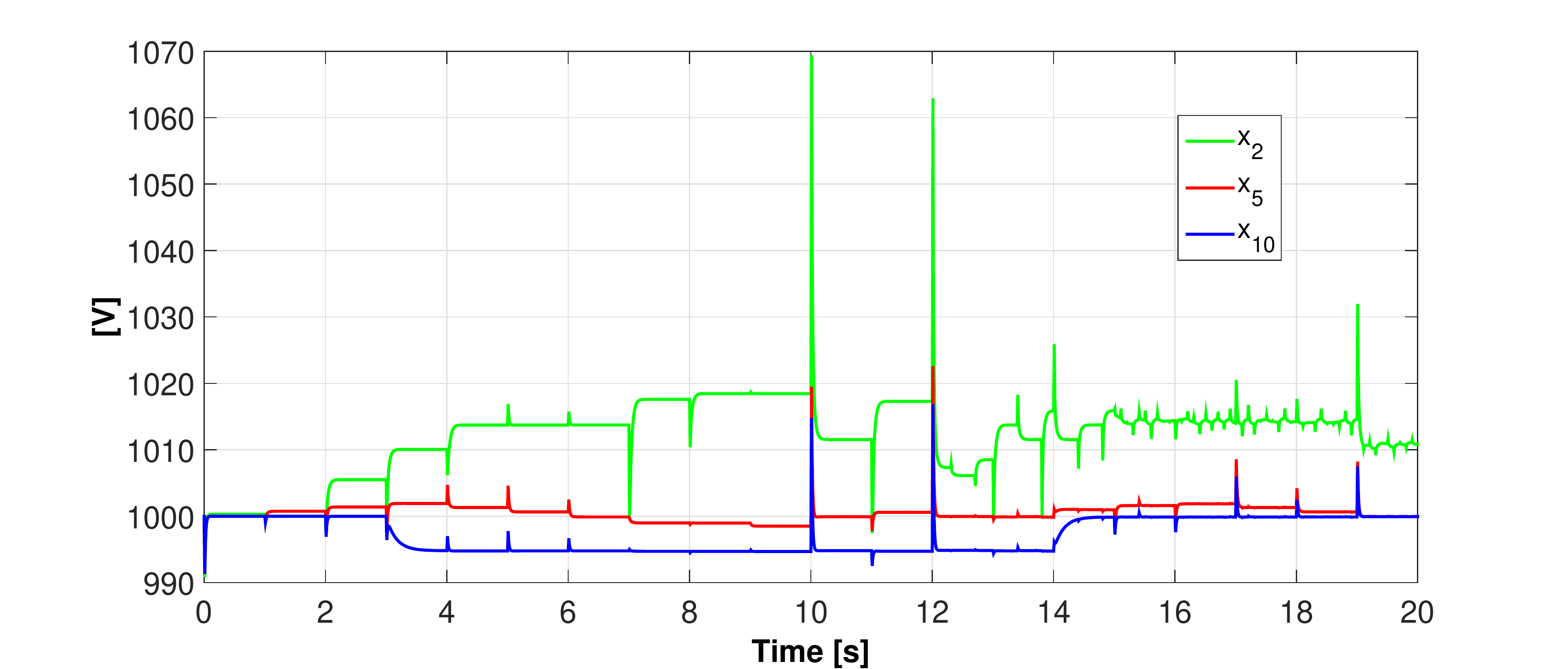}\\
   \vspace{-0.3cm}
  \caption[The voltages of ${C_2}$ and ${C_5}$]{The voltages of ${C_2}$ (green line), ${C_5}$ (red line) and $C_{10}$ (blue line).}\label{Fig_PV_BAT_SC_voltagesx2x5}
  \vspace{-0.5cm}
\end{figure}
%\vspace{-1cm}
%
%\vspace{-1cm}
\begin{figure}[h!]
%\vspace{-0.2cm}
  \centering
  \includegraphics[width=1\columnwidth]{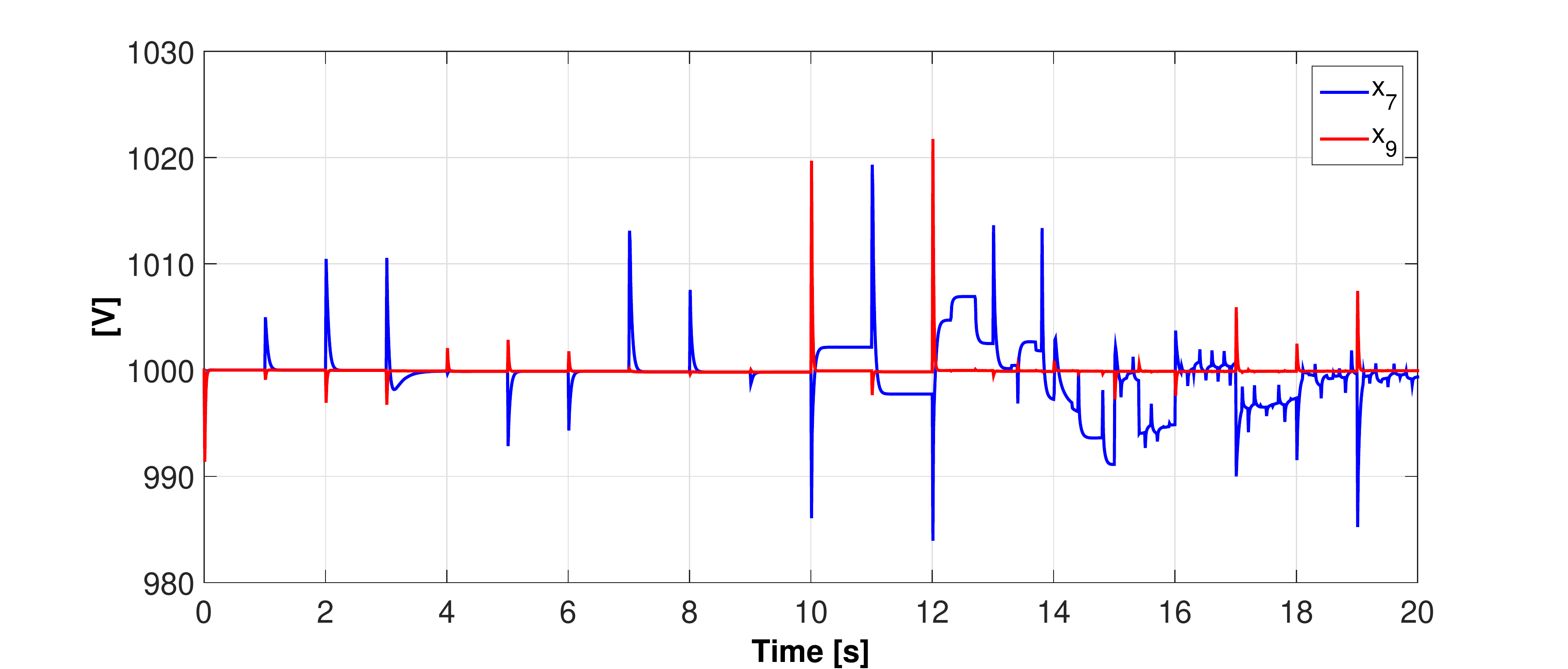}\\
    \vspace{-0.3cm}
  \caption[The voltages of ${C_7}$ and ${C_9}$]{The voltages of ${C_7}$ (blue line) and ${C_9}$ (red line).}\label{Fig_PV_BAT_SC_currentsx7x9}
  \vspace{-0.4cm}
\end{figure}
%\vspace{-0.2cm}
%

%\vspace{-1cm}
\begin{figure}[h!]
%\vspace{-0.2cm}
  \centering
  \includegraphics[width=1\columnwidth]{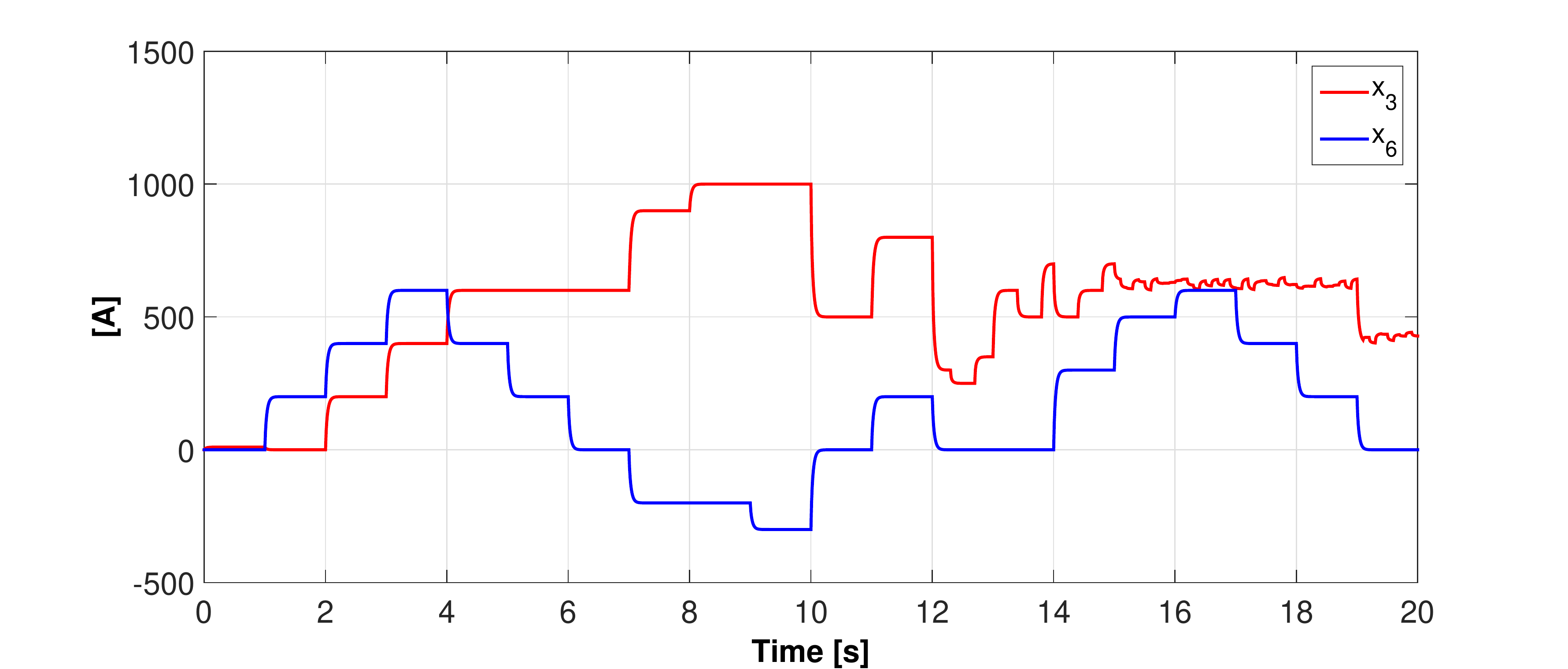}\\
    \vspace{-0.3cm}
  \caption{The currents in the inductances $L_3$ (red line) and $L_6$ (blue line).}\label{Fig_PV_BAT_SC_currents}
  \vspace{-0.3cm}
\end{figure}
%\vspace{-1cm}
%
%\vspace{-1cm}
\begin{figure}[h!]
\vspace{-0.2cm}
  \centering
  \includegraphics[width=1\columnwidth]{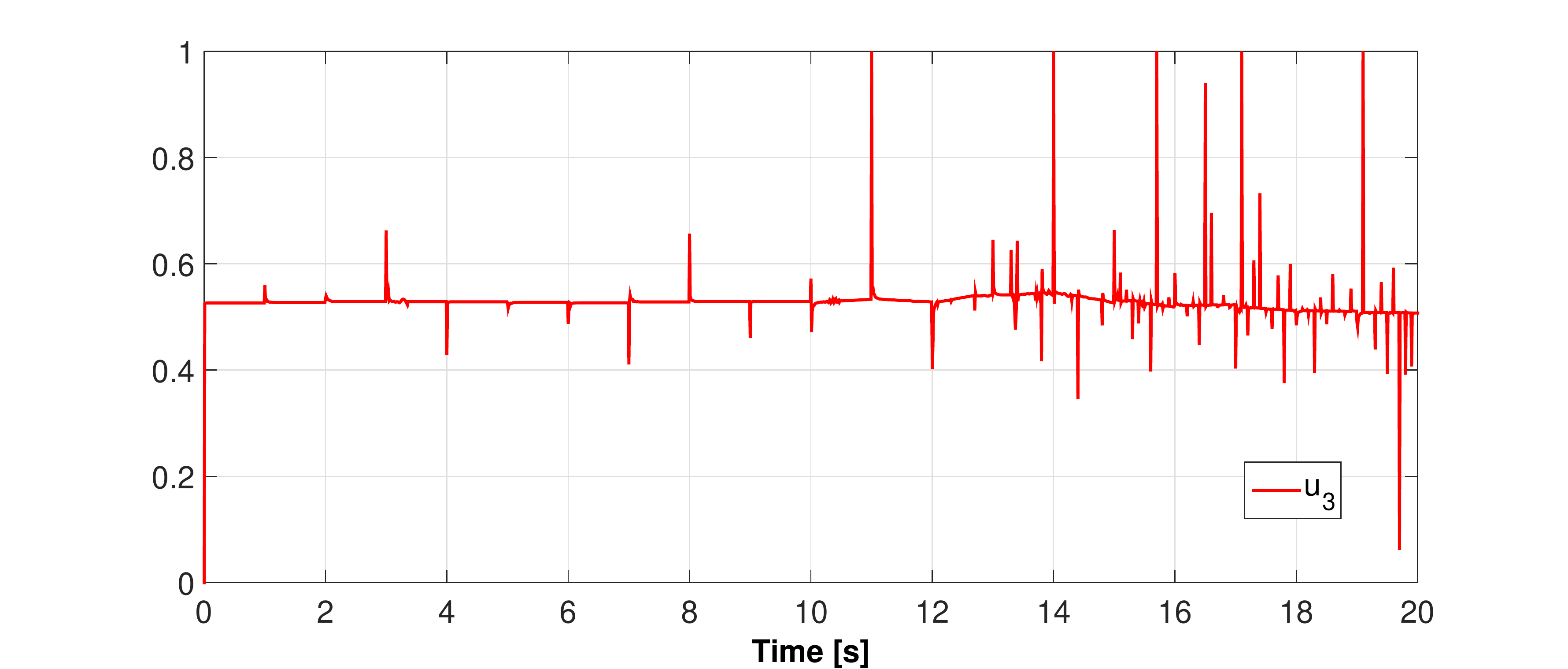}\\
    \vspace{-0.3cm}
  \caption{The control input $u_3$.}\label{Fig_PV_BAT_SC_controlu3}
  \vspace{-0.4cm}
\end{figure}
%\vspace{-1cm}
%

Figure \ref{Fig_PV_BAT_SC_currents} describes the current coming from the PV and the one related to the charge/discharge mode of the  battery: these dynamics are dependent on the voltages and related to them. The control input for the DC/DC converter connected to the supercapacitor is depicted in Figure \ref{Fig_PV_BAT_SC_controlu3}: its variation depends on the variations of voltage $V_S$ (see Figure \ref{Fig_PV_BAT_SC_voltagesVs}) and of the load (see Figures  \ref{Fig_voltagesx11x12} and \ref{Fig_PV_BAT_SC_currentsx8iL}).

Finally, we can state that the desired target to maintain DC grid voltage stability while providing a certain amount of power to a load composed by a controllable part and an uncontrollable one is ensured with the proposed control laws. The simulation describe result validity in a range of situations.

\vspace{-0.2cm}
\section{Conclusions}\label{sec_conclusion}
\vspace{-0.1cm}
In this paper the connection of a DC microgrid with the main AC grid is addressed. A DC grid composed by a renewable source and two kinds of storages, acting as energy and power reservoirs, is considered to solve the problem of correctly feeding a load; such load is composed by an uncontrollable part and a controllable one, which is the power provided to the AC grid. The modeling of the resulting grid composed by the dedicated connected devices is presented and control laws are obtained to properly satisfy requirements of voltage stability and power balance.
Formal conditions are introduced to describe the problem and rigorous analysis are carried out to obtain stability results. Simulations show that the proposed control action successfully fits the desired target to feed the load while keeping the voltage of the DC grid at a desired value.

\vspace{-0.5cm}
%\footnotesize

%\nocite{B_isidori_1995} \nocite{B_khalil_2002} \nocite{B_kundur_balu_lauby_1994}
%\nocite{A_chen_damm_cdc_2014} \nocite{Lifshitz2015Battery}
%\nocite{Olivares2014Trendsmicrogrid}
%\nocite{Hamache2014backstepping}
\footnotesize
\bibliographystyle{ieeetr}
\bibliography{mcnbib_traffic}

%\normalsize
\end{document}